\documentclass[preprint,12pt]{elsarticle}

\usepackage{amssymb,amsmath,amsthm,lineno,bm,graphicx,multirow,algorithm,algpseudocode,gensymb,textcomp,rotating}
\newtheorem{proposition}{Proposition}

\usepackage{hyperref}
\hypersetup{colorlinks=true,
            linkcolor=blue,
            anchorcolor=blue,
            citecolor=blue}

\journal{Journal of Computational Physics}

\begin{document}

\begin{frontmatter}

\title{An asymptotic-preserving reduced-order method for parametrised rarefied gas flow by proper generalised decomposition}

\author[label1]{Luowei Yin}

\author[label1,label2]{Wei Su\corref{cor1}} 

\affiliation[label1]{organization={Department of Mathematics, The Hong Kong University of Science and Technology},
            addressline={Clear Water Bay, Kowloon}, 
            city={Hong Kong},
            country={China}}
            
\affiliation[label2]{organization={Division of Emerging Interdisciplinary Areas, The Hong Kong University of Science and Technology},
            addressline={Clear Water Bay, Kowloon}, 
            city={Hong Kong},
            country={China}}


\cortext[cor1]{weisu@ust.hk}

\begin{abstract}
Modelling rarefied gas flow using the Boltzmann equation is vital in many areas. Because of the high dimensionality and coexistence of multiple characteristic scales, conventional solution strategies for this equation incur prohibitively high computational costs and are inadequate for rapid response in engineering design simulations. Based on proper generalised decomposition (PGD), we propose an \textit{a priori}, asymptotic-preserving reduced-order method for the high-dimensional parameterised Shakhov kinetic model equation. Separated representations reduce the kinetic problem to a set of low-dimensional problems. We incorporate synthetic equations into the PGD iteration so that, in the hydrodynamic limit, the iteration reduces to solving the Navier–Stokes equations, whose solution naturally exhibits a low-rank structure. Furthermore, treating the rarefaction parameter as an additional coordinate yields a solution over the whole range of rarefaction in one offline computation. It enables fast online queries to any point in the parameter space. Numerical examples using this method show reasonable precision and a significant reduction in computational cost.
\end{abstract}




\end{frontmatter}


\section{Introduction}
\label{sec1}

Kinetic theory has proven practical for describing the dynamics of rarefied gas flows in engineering applications such as microelectromechanical systems, high-altitude flight, unconventional natural gas production, and extreme ultraviolet lithography. The theory centres on the Boltzmann equation, which determines the thermofluid properties of a gaseous system by describing the evolution of the gas-molecule velocity distribution~\cite {Harris1971}. The Boltzmann equation for a monatomic gas reads as follows:
\begin{equation*}
    \frac{\partial f}{\partial t}+\bm{v}'\cdot\frac{\partial f}{\partial\bm{x}'}+\bm{F}\cdot\frac{\partial f}{\partial\bm{v}'}=\mathcal{C}\left(f,f\right),
\end{equation*}
where $f\left(t,\bm{x}',\bm{v}'\right)$ is the one-particle velocity distribution function, which is a function of time $t$, spatial position $\bm{x}'$ and molecular velocity $\bm{v}'$; $\bm{F}$ is an external force acceleration; and $\mathcal{C}$ is the Boltzmann collision integral operator, representing the variation rate of the velocity distribution function due to molecular collisions. Methods for solving the Boltzmann equation are generally stochastic~\cite{Bird1994} or deterministic~\cite{Dimarco2014}. Both methods are computationally expensive in time and memory due to the high dimensionality and multiscale nature of the equation. 
To solve the equation deterministically, one first discretises the velocity space by $N_{\text{v}}$ discrete points, resulting in $N_{\text{v}}$ partial differential equations (PDEs) in continuous time and spatial space that can be solved using conventional techniques of computational fluid dynamics. Typically, $N_{\text{v}}$ is around $10^4\sim10^6$ for a satisfactory solution, leading to a large number of degrees of freedom (DoF). Meanwhile, a stochastic method must track many particles in space and time and deal with statistical noise. In practical circumstances, rarefied gas flows often occur on multiple spatio-temporal scales, characterised by a significant separation of the Knudsen number, $\mathtt{Kn}$, which is defined as the ratio of the mean free path or the mean free time of gas molecules to a relevant process length or time scale. Therefore, parametric analysis is necessary to simulate gas flows over a wide range of Knudsen numbers. Because of the complexity of the Boltzmann collision operator, an iterative (time-marching) scheme is required even for stationary and linear problems~\cite{ADAMS20023}. Furthermore, when collisions are frequent, the collision operator induces stiffness that forces minuscule time steps for stability and accuracy, particularly to capture the correct hydrodynamic limit. This requires an asymptotic-preserving (AP) kinetic scheme. That is, a numerical scheme of the Boltzmann equation can automatically reduce to a solver for some hydrodynamic equations in the corresponding limit under different scaling~\cite{Guo2023UP}, e.g. the Euler equations for $\mathtt{Kn}\rightarrow0$ and the Navier-Stokes equations for $0<\mathtt{Kn}\ll1$ under hyperbolic scaling, without resolving the molecular mean free path/time. Several AP strategies have been developed~\cite{HU2017103}, such as penalty methods~\cite{FILBET20107625}, micro-macro decomposition~\cite{MMD2008}, implicit--explicit scheme~\cite{IMEX2013}, unified gas-kinetic framework~\cite{UGKS2026}, and synthetic iteration~\cite{SU2020109245}, each formulating a special treatment added computational complexity. Therefore, obtaining solutions for varying Knudsen numbers is costly because it involves many expensive, individual simulations.

Novel, efficient solution strategies are needed to speed up engineering simulation and design, where parametric analysis and optimisation loops often require a fast response. In recent years, reduced-order modelling has gained increasing interest for reducing the computational complexity of high-dimensional and/or parametrised PDEs, where parameters can be treated as additional coordinates~\cite{Quarteroni2013}. The basic idea is to project a solution into a low-dimensional manifold spanned by a reduced basis constructed by the underlying algorithms. The resulting solution is low-rank, in contrast to the full-rank solution. A classic example of model order reduction is digital image compression via singular value decomposition (SVD), which retains the most important features of a matrix. Depending on `when' and `how' the solution manifold is constructed, reduced-order methods are widely classified into \textit{a posteriori} and \textit{a priori} ones. \textit{A posteriori} methods, such as proper orthogonal decomposition~\cite{Berkooz1993} and the classic reduced basis method~\cite{Rozza2008}, extract low-dimensional representations from data snapshots generated by full-order solutions of the underlying physical laws across a range of parameters. In contrast, \textit{a priori} methods build basis functions on the fly directly from the governing system without requiring precomputed high-fidelity data and/or a learning phase; therefore, they are particularly attractive when computing even one such solution is intractable. Representative \textit{a priori} methods include proper generalised decomposition (PGD)~\cite{PGD2011}, dynamical low-rank (DLR) approximation~\cite{Koch2007}, and the step-and-truncate (SAT) approach~\cite{SAT2022}, which feature model-order reduction via a variable separated form: the low-rank solution is approximated as a sum of products of functions, called \emph{modes}, with each function depending on fewer coordinates of a high-dimensional problem. The number of function products is called \emph{rank}, which is usually small. As a result, DoF growth with problem dimension drops from exponential to at most linear, breaking the curse of dimensionality. The PGD computational process comprises two stages: an \textit{offline} phase uses a greedy approach to enrich modes, once and for all, providing a generalised solution (or the so-called \emph{computational vademecum}) containing all possible solutions for a prescribed range of coordinates; and an \textit{online} phase that accesses a specific solution via a simple evaluation of the separated representation, enabling a fast response in parametric analysis. DLR prescribes a rank and then dynamically evolves all modes on the time-dependent tangential plane of the low-rank solution space. SAT applies to time-dependent equations using conventional time-marching schemes; at each time step, it performs a low-rank projection followed by SVD truncation to adapt the rank. {\color{blue}Beyond low-rank approximation, some discretisation approaches can reduce model order \textit{a priori}, e.g. the sparse grid method~\cite{Bungartz_Griebel_2004}, which relies on a hierarchical multi-level basis and selects grid points to represent a high-dimensional function with fewer DoF.} 

DLR and SAT have been applied to a range of time-dependent kinetic equations~\cite{EINKEMMER2025114191}, including Vlasov-Poisson/-Maxwell equations~\cite{Einkemmer2018, EINKEMMER2020109063, GUO2022111089, GuoConservative2024}, radiative transfer equation\cite{PENG2020109735, EINKEMMER2021110353}, proton transport equation~\cite{STAMMER2026114879}, Forkker-Planck model~\cite{ELKAHZA2024113332} and Boltzmann equation for rarefied gases~\cite{Einkemmer2019, Einkemmer2021, BoltzmannDLRA, WANG2026114884, GSISDLRA}. {\color{blue}The sparse grid method has also been used for kinetic equations~\cite{Guo2016,TAO2019100022,POLLINGER2023112338,SCHNAKE2024113053}}. DLR and SAT predominantly decompose the phase space and solve lower-dimensional functions depending on $(t, \bm{x}')$ or $(t,\bm{v}')$. Nevertheless, they have not addressed parameterised problems in multi-query applications; meanwhile, dynamically evolving the reduced basis over time may not be necessary for steady-state problems. PGD has been used in various areas, such structural optimisation~\cite{structral}, geometry parametrisation~\cite{SEVILLA2020112631}, inverse parameter identification~\cite{NADAL2015113}, biomechanics~\cite{Zou2018}, power supply systems~\cite{Power}, fluid dynamics~\cite{TSIOLAKIS2022110802}. Applications to kinetic equations remain scarce. Early efforts used PGD to solve Fokker-Planck models in rheology~\cite{CHINESTA2011578}. The first phase-space-separated PGD implementation for the Boltzmann or Vlasov-Poisson equation was restricted to the transport operator~\cite{chinesta2015,EHRLACHER2017285}. More recently, PGD was extended to the stationary neutron transport equation~\cite{DOMINESEY2023112137,DOMINESEY2025111162}.

For any reduced-order modelling, two critical questions arise: whether the true solution exhibits a low-rank structure; whether the numerical method can capture it. The first question depends solely on the governing equations and the underlying physics being studied. Using Fourier analysis, it has been shown that the solution of the linear kinetic equation admits a low-rank approximation~\cite{EINKEMMER2020109063,EINKEMMER2021110353}. Nevertheless, the linear regime is only a sufficient condition. Reduced bases can represent a wide range of functions with small rank. This is consistent with the fact that, using Hilbert or Chapman-Enskog expansions~\cite{hilbert1924grundzuge,chapman1990mathematical}, the solution of a kinetic equation is expressible as polynomials of molecular velocity and some low-order derivatives of flow properties at small to moderate Knudsen numbers, which lead to asymptotics, such as the diffusion equation, the Euler or Navier-Stokes equations, and Grad's moment equations. That is, the asymptotic behaviour of a kinetic solution is naturally low rank. Therefore, a low-rank method is required to preserve the AP behaviour of a kinetic equation. DLR for the linear radiation transport equation is AP in diffusive scaling as $\mathtt{Kn}\rightarrow0$, but requires strong constraints on time and spatial discretisation and a fully implicit treatment~\cite{Ding2021}. Some strategies could help reduce the cost of AP. For example, macro-micro decomposition is combined in the solver, and the low-rank method is applied only to the micro part of the solution. The scheme eventually captures the rank-one solution in the diffusive regime~\cite{EINKEMMER2021110353}. Another AP DLR scheme is developed for the nonlinear Boltzmann equation of gas dynamics, which recovers the Euler asymptotic limit by a penalty method~\cite{EINKEMMER2025114112}.    

In this work, the authors explore whether the PGD method can serve as an efficient solver for the parametrised Boltzmann equation. The major contributions include: (1) PGD separated representations with low-dimensional problems that remain solvable on unstructured meshes of arbitrary geometries, including flows parametrised by the rarefaction parameter; (2) an AP PGD iteration process, where the scheme can approximate the correct solution when $\mathtt{Kn}\ll1$, because it becomes a macroscopic solver for the Navier-Stokes equations. To the best of our knowledge, existing DLR kinetic solvers address only square or cubic geometries and structured meshes~\cite{PENG2023111748}.   

The remainder of the paper is arranged as follows. The rarefied gas flow and its kinetic description considered in this work and the asymptotic behaviour are presented in \S\ref{Sec:Statement}. \S\ref{sec:SIS} summarises the synthetic iterative scheme, which is the basis for the AP PGD low-rank solver. The separated representation of the low-rank solution and the PGD algorithm to find the solution are detailed in \S\ref{Sec:PGD}, where a proof of AP is presented. The numerical results are given in \S\ref{Sec:Result}. The manuscript is closed with some conclusions in \S\ref{Sec:conclusion}.

\section{Statement of the Problem}\label{Sec:Statement}

\subsection{Flow in a microchannel}
Consider a rarefied gas flowing through a straight channel along the $z'$-axis with an arbitrary cross section in the $x'$-$y'$ coordinate plane. The channel connects two reservoirs containing the same gas, maintaining pressure $P_1$ and temperature $T_1$ in one and $P_2$ and $T_2$ in the other, with $P_1\leq P_2$ and $T_1<T_2$. Due to the pressure and temperature gradients, the gas flow combines a Poiseuille flow from the high- to the low-pressure reservoir and a thermal creep flow from the cold to the hot reservoir. The end effect can be neglected when the channel length is sufficiently larger than its cross-sectional dimension $H$. Furthermore, dimensionless local pressure and temperature gradients in any cross section along the flow direction $z'$ are always much lower than 1~\cite{Sharipov1997}; that is,
\begin{equation}
    \xi_P=\frac{H}{P}\frac{\mathrm{d}P}{\mathrm{d}z'}\ll1,\quad \xi_T=\frac{H}{T}\frac{\mathrm{d}T}{\mathrm{d}z'}\ll1,
\end{equation}
where $P=P(z')$ and $T=T(z')$ are the local pressure and temperature, respectively. Due to the small gradients, the velocity distribution function to describe the dynamics of the gas system can be linearised as
\begin{equation}
    f=\frac{n}{v^3_{\text{m}}}\left\{f_{\text{eq}}+\xi_P\left[h_P+zf_{\text{eq}}\right]+\xi_T\left[h_T+z\left(\bm{v}^2-\frac{5}{2}\right)f_{\text{eq}}\right]\right\},
\end{equation}
where $n$ is the molecular number density and $v_{\text{m}}=\sqrt{2k_{\text{B}}T/m}$ is the most probable molecular velocity, with $k_{\text{B}}$ being the Boltzmann constant and $m$ the molecular mass. $h_a=h_a\left(\bm{x},\bm{v}\right)$, $a=\{P,T\}$ are dimensionless perturbed velocity distribution functions responsible for the pressure and temperature gradients, respectively. Note that $\bm{x}=\left(x,y\right)$ are the two-dimensional spatial coordinates normalised by $H$ and $\bm{v}=\left(v_x,v_y,v_z\right)$ is the molecular velocity normalised by $v_{\text{m}}$. $f_{\text{eq}}$ is the dimensionless global equilibrium distribution given by
\begin{equation}
    f_{\text{eq}}=\frac{1}{\pi^{3/2}}\exp\left(-\bm{v}^2\right).
\end{equation}

It is assumed that variations of $h_a$ are governed by the steady-state Shakhov kinetic model equation subjected to the fully diffuse boundary condition, resulting in a first-order PDE problem: find $h_a(\bm{x},\bm{v})$ such that
\begin{equation}\label{Shakhov}
\begin{aligned}
    \nabla\cdot\bm{v}h_a=\delta\left[2u_av_zf_{\text{eq}}+\frac{4}{15}q_av_z\left(\bm{v}^2-\frac{5}{2}\right)f_{\text{eq}}-h_a\right]-s_a,\quad \bm{x}\in\Omega,\\
    h_a=0,\quad \bm{x}\in\partial\Omega,\ \bm{v}\cdot\bm{n}<0,
\end{aligned}
\end{equation}
where
\begin{equation*}
    s_a=\begin{cases}
        v_zf_{\text{eq}},&a=P,\\
        v_z\left(\bm{v}^2-\dfrac{5}{2}\right)f_{\text{eq}},\quad&a=T.
    \end{cases}
\end{equation*}
In Eq.~\eqref{Shakhov}, $\nabla=(\partial_x,\partial_y)$; $\Omega$ represents the region of the channel cross section with boundary $\partial\Omega$; $\bm{n}$ is the outward unit normal vector at the boundary, and {\color{magenta}the inflow value of $h_a$ is always zero for this specific flow~\cite{Titarev2010}}. $\delta$ is the rarefaction parameter, which is inversely proportional to the Knudsen number
\begin{equation*}
    \delta=\frac{PH}{\mu_g v_{\text{m}}}\sim\frac{1}{\mathtt{Kn}},
\end{equation*}
where $\mu_g$ is the gas viscosity. At the macroscale, dimensionless flow velocity $u_{a}$ and heat flux $q_{a}$ are obtained by taking velocity moments of the perturbed distribution function, such as
\begin{equation}\label{moments}
    u_{a}=\langle v_z,h_a\rangle_{\bm{v}},\quad q_{a}=\langle v_z\left(\bm{v}^2-5/2\right),h_a\rangle_{\bm{v}},
\end{equation}
with $\langle\cdot,\cdot\rangle_{\bm{v}}$ denoting the $L^2$ inner product in the molecular velocity space. Dimensionless flow rates are also of interest. The one for the Poiseuille flow is computed as
\begin{equation}
    G_P=-2\int_{\Omega} u_P\mathrm{d}\bm{x},
\end{equation}
while that of the thermal creep flow can be obtained according to the Onsager-Casimir relation~\cite{Sharipov1999}
\begin{equation}
    G_T=2\int_{\Omega} q_P\mathrm{d}\bm{x},
\end{equation}
allowing calculation of the Poiseuille flow only. In the following sections, $s_a$ is chosen as $s_a=v_zf_{\text{eq}}$ and the subscript $a=\{P,T\}$ will be omitted.  

\subsection{Low-rank structure in hydrodynamic regime}\label{subsec:hydro}

In the Navier-Stokes limit, i.e., when the Knudsen number is sufficiently small (but not necessarily zero), the low-rank structure of the solution to Eq.~\eqref{Shakhov} can be shown using the Chapman-Enskog expansion~\cite{chapman1990mathematical}. This asymptotic analysis starts from the time-dependent version of Eq.~\eqref{Shakhov}, that is, 
\begin{equation}\label{ShakhovTime}
    \partial_th+\nabla\cdot\bm{v}h=\delta\left[2uv_zf_{\text{eq}}+\frac{4}{15}qv_z\left(\bm{v}^2-\frac{5}{2}\right)f_{\text{eq}}-h\right]-s,
\end{equation}
where $h$ and $\partial_t$ are expressed in a power series of the Knudsen number
\begin{equation}\label{CE}
\begin{aligned}
    h&=h^{(0)}+\delta^{-1}h^{(1)}+\delta^{-2}h^{(2)}+\cdots,\\
    \partial_t&=\partial_{t_0}+\delta^{-1}\partial_{t_1}+\delta^{-2}\partial_{t_2}+\cdots,
\end{aligned}
\end{equation}
where $\partial_{t_i}$ denotes the contribution to $\partial_t$ from the derivatives of the flow properties $u$ and $q$~\cite{Guo2023UP}. The coefficients $h^{(i)}$ are of $O(1)$ and are found by substituting the expansion~\eqref{CE} into Eq.~\eqref{Shakhov}, multiplying $\delta^{-1}$ on both sides and collecting the terms corresponding to each level of $\delta^{-1}$. At the Navier-Stokes limit, the solution is truncated to $O(\delta^{-1})$, which is obtained by the following three balance equations
\begin{align}
    \delta^0:&\quad 2uv_zf_{\text{eq}}+\frac{4}{15}qv_z\left(\bm{v}^2-\frac{5}{2}\right)f_{\text{eq}}-h^{(0)}=0,\label{h0}\\
    \delta^{-1}:&\quad \partial_{t_0}h^{(0)}+\nabla\cdot\bm{v}h^{(0)}=-h^{(1)}-s,\label{h1}\\
    \delta^{-2}:&\quad \partial_{t_1}h^{(0)}+\partial_{t_0}h^{(1)}+\nabla\cdot\bm{v}h^{(1)}=-h^{(2)}.\label{h2}
\end{align}
Eq.\eqref{h0} gives $h^{(0)}=2uv_zf_{\text{eq}}$. Note that taking $\langle v_z\left(\bm{v}^2-5/2\right),\cdot\rangle_{\bm{v}}$ to Eq.\eqref{h0} yields $q=q/3$, which implies that the heat flux is zero when $\delta^{-1}\rightarrow0$. Meanwhile, taking $\langle v_z,\cdot\rangle_{\bm{v}}$ yields $\langle v_z,h^{(0)}\rangle_{\bm{v}}=u=\langle v_z,h\rangle_{\bm{v}}$, which is identical to $\sum_{i=1}\langle v_z,h^{(i)}\rangle_{\bm{v}}=0$. This follows from momentum conservation in the collision term. We can require $\langle v_z,h^{(i)}\rangle_{\bm{v}}=0$, for $i\geq1$. Substituting $h^{(0)}$ into Eq.~\eqref{h1}, $h^{(1)}$ is obtained as
\begin{equation}
    h^{(1)}=-2v_zf_{\text{eq}}\partial_{t_0}u-2v_zf_{\text{eq}}\bm{v}\cdot\nabla u-s,
\end{equation}
where $\partial_{t_0}u$ is found taking $\langle v_z,\cdot\rangle_{\bm{v}}$ to Eq.~\eqref{h1}, giving $\partial_{t_0}u=-1/2$. Eventually,
\begin{equation}\label{h1solution}
    h^{(1)}=-2v_zf_{\text{eq}}\bm{v}\cdot\nabla u,
\end{equation}
Therefore, when the Knudsen number is small, i.e., $\delta^{-1}\ll 1$,
\begin{equation}
    h\simeq h^{(0)}+\delta^{-1}h^{(1)}=2uv_zf_{\text{eq}}-2\delta^{-1}v_xv_zf_{\text{eq}}\partial_xu-2\delta^{-1}v_yv_zf_{\text{eq}}\partial_yu,
\end{equation}
This approximation has an intrinsic low-rank structure of rank 3, since each term is expressible in separated coordinates. The governing equation for $u$ is obtained by taking $\langle v_z,\cdot\rangle_{\bm{v}}$ to Eq.~\eqref{h1}, taking $\langle \delta^{-1}v_z,\cdot\rangle_{\bm{v}}$ to Eq.~\eqref{h2} and adding the resulting equations together, read as
\begin{equation}\label{uNS}
    \partial_tu+\frac{1}{2}\partial_x p_{xz}+\frac{1}{2}\partial_y p_{yz}=-\frac{1}{2},
\end{equation}
where $p_{xz}=\delta^{-1}\langle2v_xv_z,h^{(1)}\rangle_{\bm{v}}$ and $p_{yz}=\delta^{-1}\langle2v_yv_z,h^{(1)}\rangle_{\bm{v}}$ are the shear stresses. The constitutive laws for $p_{xz}$, $p_{yz}$, as well as $q$ are obtained from ~\eqref{h1solution}, written as
\begin{equation}\label{constutive}
    p_{xz}=-\delta^{-1}\partial_x u,\quad p_{yz}=-\delta^{-1}\partial_y u,\quad q=0.
\end{equation}
For the stationary case, combining Eqs.~\eqref{uNS} and~\eqref{constutive} gives
\begin{equation}\label{udiff}
    \nabla^2u=\delta,\quad q=0.
\end{equation}
These are the simplified, time-independent Navier-Stokes equations for the flow considered. The heat flux is always zero because Poiseuille flow is isothermal. The analytical solution for the stationary flow rate depends on the cross-sectional geometry. For a square channel, the dimensionless flow rates are~\cite{Titarev2010}
\begin{equation}\label{GpHydro}
    G_P=\frac{\delta}{6}\left[1-\frac{192}{\pi^5}\sum^{\infty}_{n=0}\frac{\tanh\left(0.5\pi\left(2n+1\right)\right)}{\left(2n+1\right)^5}\right],\quad G_T=0.
\end{equation}

\subsection{Low-rank structure in free-molecular regime}\label{subsec:free}

In the free-molecular regime, when $\delta=0$, the kinetic equation can be integrated analytically by the method of characteristics. For a square with unit sides and its vertices at $(0,0)$, $(1,0)$, $(1,1)$ and $(0,1)$, the solution of $h$ may have one of the following forms, depending on the direction of characteristics~\cite{Titarev2010},
\begin{equation}
    h=s\frac{x-1}{v_x},\quad s\frac{x}{v_x},\quad s\frac{y-1}{v_y},\quad s\frac{y}{v_y}.
\end{equation}
The solution is also low-rank. The flow rates are~\cite{Sharipov1999}
\begin{equation}\label{GpGtfree}
    G_P=\frac{1}{\sqrt{\pi}}\left[2\ln\left(1+\sqrt{2}\right)-\frac{2}{3\left(1+\sqrt{2}\right)}\right],\quad G_T=\frac{1}{2}G_P.
\end{equation}

\section{Synthetic Iterative Scheme}\label{sec:SIS}

Before describing the AP PGD method for solving rarefied gas flows, we first summarise the synthetic iterative scheme~\cite{Valougeorgis2003,SU2020109245} for the full-rank solution. Given some initial data $\{h^n,u^n,q^n\}_{n=0}$, the solution $h^{n+1}$ in the $(n+1)$-th step is solved from
\begin{equation}\label{ShakhovSIS}
\begin{aligned}
    \nabla\cdot\bm{v}h^{n+1}=\delta\left[2u^nv_zf_{\text{eq}}+\frac{4}{15}q^nv_z\left(\bm{v}^2-\frac{5}{2}\right)f_{\text{eq}}-h^{n+1}\right]-s,\ \bm{x}\in\Omega,\\
    h^{n+1}=0,\quad \bm{x}\in\partial\Omega,\ \bm{v}\cdot\bm{n}<0,
\end{aligned}
\end{equation}
where $u^n$ and $q^n$ are the solutions of the following macroscopic equations calculated at the $n$-th step~\cite{Valougeorgis2003}
\begin{equation}\label{macroEq}
    \begin{aligned}
    \nabla^2u^n=&\delta-\partial_{xx}\langle\phi_1, h^n\rangle_{\bm{v}}-\partial_{xy}\langle2\phi_2,h^n\rangle_{\bm{v}}-\partial_{yy}\langle\phi_3, h^n\rangle_{\bm{v}},&\bm{x}\in\Omega,\\
    \nabla^2q^n=&\frac{4}{3}\delta^2q^n-\partial_{xx}\langle \psi_1 ,h^n\rangle_{\bm{v}}-\partial_{xy}\langle 2\psi_2, h^n\rangle_{\bm{v}}-\partial_{yy}\langle\psi_3,h^n\rangle_{\bm{v}},&\bm{x}\in\Omega,\\
    u^n=&u^n_h=\langle v_z,h^n\rangle_{\bm{v}},\quad q^n=q^n_h=\Big\langle v_z\left(\bm{v}^2-\frac{5}{2}\right),h^n\Big\rangle_{\bm{v}},&\bm{x}\in\partial\Omega,
    \end{aligned}
\end{equation}
with
\begin{equation}
    \begin{aligned}
        \{\phi_1,\ \phi_2,\ \phi_3\}&=\left\{v_z\left(2v^2_x-1\right),\ 2v_xv_yv_z,\ v_z\left(2v^2_y-1\right)\right\},\\
        \{\psi_1,\ \psi_2,\ \psi_3\}&=\left(\bm{v}^2-\dfrac{5}{2}\right) \{\phi_1,\ \phi_2,\ \phi_3\}.
    \end{aligned}
\end{equation}
Note that the Dirichlet boundary condition is used, where the boundary values are calculated by taking velocity moments of $h^n$. $u_h$ and $q_h$ are used to emphasise that the values are calculated from the distribution function. 

For deterministic solutions of the kinetic equation~\eqref{ShakhovSIS}, the so-called discrete velocity method is commonly used, where the continuous molecular velocity is replaced by $N_{\text{v}}=N_x\times N_y\times N_z$ discrete points $\bm{v}^{j}=\left(v^{j_1}_x,v^{j_2}_y,v^{j_3}_z\right)$ with $j_1=\{1,\dots,N_x\}$, $j_2=\{1,\dots,N_y\}$, and $j_3=\{1,\dots,N_z\}$. Then, the governing equations for the unknown distribution functions $h^{j,n+1}=h^{n+1}\left(\bm{x},\bm{v}^{j}\right)$ at discrete velocities are solved using a standard numerical scheme for first-order PDEs. {\color{blue}In this work, the nodal discontinuous Galerkin finite element method employing upwind numerical flux is used on unstructured triangular meshes. The transport sweep technique is used to solve the advection operator implicitly by marching along the direction of characteristics from a given boundary condition, avoiding the assembly of a global matrix~\cite{Pautz01022002}.} The discrete distribution functions are approximated in the piecewise polynomial space of complete degree at most $k$, the total DoF to approximate Eq.~\eqref{ShakhovSIS} is $N_{\text{v}}\times N_{\text{DG}}$, where $N_{\text{DG}}=N_{\text{el}}\left(k+1\right)\left(k+2\right)/2$ with $N_{\text{el}}$ being the number of finite elements partitioning the computational domain $\Omega$. The macroscopic equations~\eqref{macroEq} take the form of a diffusion equation. They are solved by the continuous Galerkin method with the same polynomial degree $k$, and thus the total DoF to approximate Eq.~\eqref{macroEq} is $2N_{\text{CG}}$ with $N_{\text{CG}}\simeq N_{\text{el}}k^2/2$ for a relatively uniform mesh. The synthetic iterative scheme expedites the convergence, particularly in the small-Knudsen-number regimes. Let $N^{\text{fr}}_{\text{itr}}$ be the number of iterations for a convergent solution; the overall computational complexity is $O\left(N^{\text{fr}}_{\text{itr}}\left(N_\text{v}N_{\text{DG}}+2N_{\text{CG}}\right)\right)$, and the required memory is proportional to $O\left(N_\text{v}N_{\text{DG}}+2N_{\text{CG}}\right)$. The obtained solution is the full-rank one. 

\section{Proper Generalised Decomposition Method}\label{Sec:PGD}

This section details the proper generalised decomposition strategies for \textit{a priori} low-rank solutions of the Boltzmann equation. To achieve AP, the method solves the system of the synthetic iterative scheme, i.e., Eqs.~\eqref{ShakhovSIS} and~\eqref{macroEq}. Specifically, the low-rank solver is applied to the kinetic equation~\eqref{ShakhovSIS}.

\subsection{Separated representation of perturbed velocity distribution function}

To construct a low-rank solution to the problem \eqref{ShakhovSIS} using PGD, the molecular velocity is expressed in cylindrical coordinates as
\begin{equation}\label{cylind}
\begin{aligned}
    v_x=v_r\cos\theta,\quad v_y=v_r\sin\theta,\quad\bm{v}^2=v_r^2+v_z^2,\\ \mathrm{d}\bm{v}=v_r\mathrm{d}\theta\mathrm{d}v_r\mathrm{d}v_z:=\mathrm{d}\theta\mathrm{d}\bm{v}_{rz},
\end{aligned}
\end{equation}
where $\theta\in\left[0,2\pi\right]$ is the polar angle in the $x$-$y$ plane and $v_r=\sqrt{v^2_x+v^2_y}$. {\color{magenta}This cylindrical setting helps to deal with arbitrary geometries in the low-dimensional spatial problem presented later.} Substituting Eq.~\eqref{cylind} into Eqs.~\eqref{ShakhovSIS}, the kinetic equation becomes
\begin{equation}\label{ShakhovCylinder}
    \begin{aligned}
        v_r\nabla\cdot\bm{s}h=\delta\left[2uv_zf_{\text{eq}}+\frac{4}{15}qv_z\left(\bm{v}^2-\frac{5}{2}\right)f_{\text{eq}}-h\right]-s,\quad\bm{x}\in\Omega,\\
    h=0,\quad \bm{x}\in\partial\Omega,\ \bm{s}\cdot\bm{n}<0,
    \end{aligned}
\end{equation}
where $\bm{s}=(\cos\theta,\sin\theta)$ is the direction vector. The independent variables of the system become $x$, $y$, $v_r$, $\theta$ and $v_z$, and the value of $\bm{s}\cdot\bm{n}$ determines when to impose the boundary value of $h$.

PGD assumes that a low-rank solution is expressible in terms of a finite sum of products of functions that depend on parts of the coordinates, proposed as
\begin{equation}\label{PGDh}
    h\left(\bm{x},v_r,\theta,v_z\right)=\sum^m_{i=1}X_i\left(\bm{x},\theta\right)V_i\left(v_r,v_z\right),
\end{equation}
where $X_i$ and $V_i$ are unknown functions, named the $i$-th spatio-angular mode and the $i$-th mesoscopic velocity mode, respectively; and $m$ is the total number of PGD modes, i.e., the rank of the solution. {\color{magenta}The way of splitting coordinates is essential for solvability. Later, it will be shown that the problem of solving the modes $X_i$ is a boundary value problem with an advection operator [i.e., Eq.~\eqref{Xstrong}]. If we intend to solve it also using the discontinuous Galerkin method and upwind numerical flux (as is the case for the full-rank solution), the inflow of $X_i$ for each finite element, including the boundary condition, should be specified. Therefore, variables describing the inflow characteristic direction must be grouped with spatial coordinates. In cylindrical coordinates, only the polar angle $\theta$ is needed, and hence it is grouped with $x$ and $y$, leading to the separation in ~\eqref{PGDh}.}


The PGD modes are commonly found by the Galerkin method and are enriched sequentially~\cite{PGDprimer}, yielding a sequence of rank-one approximations for the unresolved part: when computing the $m$-th modes, the first $m-1$ modes are known. Thus, the $m$-th modes $X_m$ and $V_m$ are found such that
\begin{equation*}
    \begin{aligned}
        \int X^{\ast}V^{\ast}\Bigg\{\delta X_mV_m+v_rV_m\nabla\cdot\bm{s}X_m=2\delta uv_zf_{\text{eq}}+\frac{4}{15}\delta qv_z\left(\bm{v}^2-\frac{5}{2}\right)f_{\text{eq}}\\
        -s-\sum^{m-1}_{i=1}\left(\delta X_iV_i+v_rV_i\nabla\cdot\bm{s}X_i\right)\Bigg\}\mathrm{d}\bm{x}\mathrm{d}\bm{v},
    \end{aligned}
\end{equation*}
where $X^{\ast}V^{\ast}$ are some test functions; $u$ and $q$ are the flow properties solved from~\eqref{macroEq} when $h=\sum^m_{i=1}X_iV_i$. Although the original physical equation is linear, PGD results in a high-dimensional nonlinear problem. Fixed-point iteration with an alternating-direction strategy is commonly used to approximate the nonlinear problem~\cite {PGDprimer}. Due to the separated representation, the computational cost is reduced since the high-dimensional problem is decomposed into a few low-dimensional ones. 

\subsection{Fixed-point iteration}

In the $(n+1)$-th step of the fixed-point iteration to calculate the $m$-th modes, first, $X_m$, as well as $u$ and $q$, are frozen at the previous step, and the velocity mode $V_m$ is updated. Thus, the test function is set to $X^{\ast}=X^n_m$. $V^{n+1}_m$ is found from the following strong form
\begin{equation}\label{Vstrong}
    \begin{aligned}
        \delta\alpha_{mm}V^{n+1}_m+v_r\beta_{mm}V^{n+1}_m=2\delta&\gamma_mv_zf_{\text{eq}}+\frac{4}{15}\delta \kappa_mv_z\left(\bm{v}^2-\frac{5}{2}\right)f_{\text{eq}}\\
         &-\sigma_m v_zf_{\text{eq}}-\sum^{m-1}_{i=1}\left(\delta\alpha_{mi}V_i+v_r\beta_{mi}V_i\right),
    \end{aligned}
\end{equation}
where
\begin{equation*}
    \begin{aligned}
        \alpha_{mi}=\int X^{n}_{m}X_{i}\mathrm{d}\bm{x}\mathrm{d}\theta,\quad
        \beta_{mi}=\int\left(X^n_m\nabla\cdot\bm{s}X_i\right) \mathrm{d}\bm{x}\mathrm{d}\theta,\\
        \gamma_{m}=\int X^n_{m}u^n\mathrm{d}\bm{x}\mathrm{d}\theta,\quad\kappa_{m}=\int X^n_{m}q^n\mathrm{d}\bm{x}\mathrm{d}\theta,\quad\sigma_m=\int X^n_{m}\mathrm{d}\bm{x}\mathrm{d}\theta,\\
        X_i=X^n_m,\quad \text{for }i=m. 
    \end{aligned}
\end{equation*}

Then, the $m$-th spatio-angular mode $X_m$ is computed, where $V_m$, $u$, and $q$ are known. The test function is set to $V^{\ast}=V^{n+1}_m$. $X^{n+1}_m$ is obtained from the following boundary value problem
\begin{equation}~\label{Xstrong}
    \begin{aligned}
        \delta\hat{\alpha}_{mm} X^{n+1}_m+\hat{\beta}_{mm}\nabla\cdot\bm{s}X^{n+1}_m=2\delta\hat{\gamma}_mu^n+\frac{4}{15}\delta\hat{\kappa}_mq^n \\
       -\hat{\gamma}_m-\sum^{m-1}_{i=1}\left(\delta\hat{\alpha}_{mi} X_i+\hat{\beta}_{mi}\nabla\cdot\bm{s}X_i\right),\quad\bm{x}\in\Omega,\\
       X^{n+1}_m=0,\quad \bm{x}\in\partial\Omega,\ \bm{s}\cdot\bm{n}<0,
    \end{aligned}
\end{equation}
where
\begin{equation*}
    \begin{aligned}
        \hat{\alpha}_{mi}=\int V^{n+1}_{m}V_{i}\mathrm{d}\bm{v}_{rz},\quad \hat{\beta}_{mi}=\int v_rV^{n+1}_{m}V_{i}\mathrm{d}\bm{v}_{rz},\\
        \hat{\gamma}_m=\int v_zV^{n+1}_{m}f_{\text{eq}}\mathrm{d}\bm{v}_{rz},\quad\hat{\kappa}_m=\int v_z\left(\bm{v}^2-\frac{5}{2}\right)V^{n+1}_{m}f_{\text{eq}}\mathrm{d}\bm{v}_{rz},\\
        V_i=V^{n+1}_m,\ \text{for }i=m.
    \end{aligned}
\end{equation*}

Finally, $u$ and $q$ are solved from the macroscopic equations ~\eqref{macroEq} with updated $X_mV_m$.

\subsection{PGD algorithm and computational complexity}

A PGD algorithm is generally constructed as a two-loop structure: starting from the first mode, the outer loop enriches the PGD modes, and the inner loop applies alternating-direction fixed-point iterations to search for the velocity and spatio-angular modes. If the PGD algorithm converges, the mode amplitude generally decreases as $m$ becomes larger, implying that the contribution of successive modes decreases. To balance accuracy against computational cost, stopping criteria should be given for both the outer and inner loops. In this work, the outer loop can either be stopped by a prescribed number of modes $M_{\text{md}}$, or {\color{blue} be terminated when the following residual $\mathtt{res}_{\text{out}}$ is less than a tolerance $\mathtt{tol}_{\text{out}}$
\begin{equation}\label{outer}
\begin{aligned}
    \mathtt{res}_{\text{out}}=\min&\Bigg\{
    \frac{ \|X_m\|\cdot\|V_m\| }{\sum_{i=1}^{m-1}\|X_i\|\cdot \|V_i\|},\\
    &\max\Bigg\{\frac{\|u_m-u_{m-1}\|}{\|u_{m-1}\|},\frac{\|q_m-q_{m-1}\|}{\|q_{m-1}\|}\Bigg\}\Bigg\} <\mathtt{tol}_{\text{out}},
\end{aligned}
\end{equation}
where $u_m/q_m$ and $u_{m-1}/q_{m-1}$ represent those solved by Eq.~\eqref{macroEq} based on $h$ with the total $m$ modes and $m-1$ modes, respectively.} This means that the enrichment of the PGD mode stops when either the new modes have small amplitudes or the macroscopic quantities converge. Meanwhile, {\color{blue}the inner loop is stopped when the residual $\mathtt{res}_{\text{in}}$ is less than a tolerance $\mathtt{tol}_{\text{in}}$, i.e.,
\begin{equation}\label{inner}
    \mathtt{res}_{\text{in}}=\max\Bigg\{\frac{\|V^{n+1}_m-V^n_m\|}{\|V^n_m\|},\ \frac{\|X^{n+1}_m-X^n_m\|}{\|X^n_m\|}\Bigg\}<\mathtt{tol}_{\text{in}}.
\end{equation}}
Here, $\|\cdot\|$ is the standard $\mathcal{L}_2$ norm. The implementation of PGD is presented in \textbf{Algorithm}~\ref{Algorithm}. Note that the spatio-angular mode $X_m$ is normalised in line 9 for numerical stability~\cite{NOUY20101603}, so that the magnitude of the $m$-PGD mode is carried by $V_m$.


\begin{algorithm}[t]
\caption{PGD algorithm}\label{Algorithm}
\begin{algorithmic}[1]
\State Specify user-controlled input $\mathtt{tol}_{\text{out}}$ (and $M_{\text{md}}$), $\mathtt{tol}_{\text{in}}$
\State $m\leftarrow1$
\While{$\mathtt{res}_{\text{out}}\geq\mathtt{tol}_{\text{out}}$ (and $m\leq M_{\text{md}}$)}
    \State Initialise $X^0_m$, $V^0_m$, $u^0$, and $q^0$
    \State $n\leftarrow0$
    \While{$\mathtt{res}\geq\mathtt{tol}_{\text{in}}$}
        \State Solve Eq.~\eqref{Vstrong} for $V^{n+1}_m$
        \State Solve Eq.~\eqref{Xstrong} for $X^{n+1}_m$
        \State Normalise $X_m$: $X_m\leftarrow X_m/\| X_m\|$
        \State Calculate $\mathtt{res}_{\text{in}}$ [Eq.~\eqref{inner}]
        \State Solve Eq.~\eqref{macroEq} for $u^{n+1}$ and $q^{n+1}$
        \State $n\leftarrow n+1$
    \EndWhile
    \State Calculate $\mathtt{res}_{\text{out}}$ [Eq.~\eqref{outer}]
    \State $m\leftarrow m+1$
\EndWhile
\end{algorithmic}
\end{algorithm}

When solving Eqs.~\eqref{Vstrong} for solutions of mesoscopic velocity modes, the discrete velocity method is still applicable. Let $v_r$ and $v_z$ be discretised by points $\{v^{j_1}_r\}^{N_r}_{j_1=1}$ and $\{v^{j_2}_z\}^{N_z}_{j_2=1}$, respectively. The resultant algebraic system has a size of $N_r\times N_z$. To solve spatio-angular modes from Eq.~\eqref{Xstrong}, the polar angle is discretised by $N_{\theta}$ points $\{\theta^{j_3}\}^{N_{\theta}}_{j_3=1}$. {\color{blue}The resultant PDEs are approximated using the discontinuous Galerkin method with upwind flux, and similar to that for the full-rank solution, transport sweep along each direction $\bm{s}^{j_3}=\left(\cos\theta^{j_3},\sin\theta^{j_3}\right)$ is used to solve the advection operator implicitly.} Macroscopic equations~\eqref{macroEq} are still solved using the continuous Galerkin method. In summary, the computational complexity of the PGD solution is around $O\left(M_{\text{md}} N^{\text{in}}_{\text{itr}}\left(N_rN_z+ N_{\theta}N_{\text{DG}}+2N_{\text{CG}}\right)\right)$ and the memory requirement is proportional to $O\left(M_{\text{md}}\left(N_rN_z+ N_{\theta}N_{\text{DG}}+2N_{\text{CG}}\right)\right)$. Compared to the full-rank solution, the computational complexity and the required memory can be significantly reduced as $M_{\text{md}}$ and $N^{\text{in}}_{\text{itr}}$ are expected to be small. 

\subsection{Asymptotic preserving property}

The asymptotic behaviour of the PGD low-rank kinetic solver can be shown by performing the Chapman-Enskog expansion. Since the PGD iteration process is of core interest, the spatial and velocity coordinates remain undiscretised. The separated representation~\eqref{PGDh} is expanded by the Chapman-Enskog ansatz as
\begin{equation}\label{modeCE}
    \begin{aligned}
        h=\sum^m_{i=1}X_iV^{(0)}_i+\delta^{-1}\sum^m_{i=1}X_iV^{(1)}_i+\delta^{-2}\sum^m_{i=1}X_iV^{(2)}_i+\cdots,
    \end{aligned}
\end{equation}
where expansion is applied to $V_i$ mode since $X_i$ is always of $O(1)$ due to normalisation (see line 9 in Algorithm~\ref{Algorithm}). Substituting the expansion into the macroscopic equations~\eqref{macroEq} and truncating to $O(\delta^{-1})$, we have
\begin{equation}\label{macroApEq}
    \begin{aligned}
    \frac{1}{\delta}\nabla^2u^n=&1-\frac{1}{\delta}\partial_{xx}\langle \phi_1, X^n_mV^{(0),n}_m+\sum^{m-1}_{i=1}X_iV^{(0)}_i\rangle_{\bm{v}}\\
    &-\partial_{xy}\langle2\phi_2,X^n_mV^{(0),n}_m+\sum^{m-1}_{i=1}X_iV^{(0)}_i\rangle_{\bm{v}}\\
    &-\partial_{yy}\langle \phi_3, X^n_mV^{(0),n}_m+\sum^{m-1}_{i=1}X_iV^{(0)}_i\rangle_{\bm{v}},&\bm{x}\in\Omega,\\
    u^n=&u^{(0),n}_h,&\bm{x}\in\partial\Omega,\\
    \end{aligned}
\end{equation}
and $q^n=0$, i.e., the heat flux is always equal to zero.

\begin{proposition}\label{prop}
The PGD low-rank solver presented in Algorithm~\ref{Algorithm} becomes a macroscopic solver for the Navier-Stokes equations~\eqref{udiff} in the limit $0<\delta^{-1}\ll1$, in the sense that
\begin{enumerate}
    \item[(1)] all the terms of $\langle\cdot,\cdot\rangle_{\bm{v}}$ vanish in Eq.~\eqref{macroApEq};
    \item[(2)] $u^n=u^{(0),n+1}_h+O(\delta^{-1-\epsilon})$, with $\epsilon\geq0$.
\end{enumerate}
\end{proposition}

\begin{proof}
    Eqs.~\eqref{Vstrong} and~\eqref{Xstrong} to solve the PGD modes are rewritten using vector notations:
    \begin{equation}\label{VstrongVector}
        \mathbf{A}\cdot\mathbf{V}+\delta^{-1}v_r\mathbf{B}\cdot\mathbf{V}=2\gamma_mv_zf_{\text{eq}}+\frac{4}{15}\kappa_mv_z\left(\bm{v}^2-\frac{5}{2}\right)f_{\text{eq}}-\delta^{-1}\sigma_mv_zf_{\text{eq}},
    \end{equation}
    \begin{equation}\label{XstrongVector}
        \mathbf{\hat{A}}\cdot\mathbf{X}+\delta^{-1}\nabla\cdot\bm{s}\mathbf{\hat{B}}\cdot\mathbf{X}=2\hat{\gamma}_mu^n+\frac{4}{15}\hat{\kappa}_mq^n-\delta^{-1}\hat{\gamma}_m,
\end{equation}
where the vectors $\mathbf{V}$ and $\mathbf{X}$ collect
\begin{equation*}
    \mathbf{V}=\left[\begin{array}{cccc}
       V_1  & \cdots & V_{m-1} & V^{n+1}_m 
    \end{array}\right]^T,\quad \mathbf{X}=\left[\begin{array}{cccc}
       X_1  & \cdots & X_{m-1} & X^{n+1}_m 
    \end{array}\right]^T,\\
\end{equation*}
and the other vectors are
\begin{equation*}
    [\mathbf{A}]_i=\alpha_{mi},\quad [\mathbf{B}]_i=\beta_{mi},\quad [\mathbf{\hat{A}}]_i=\hat{\alpha}_{mi},\quad [\mathbf{\hat{B}}]_i=\hat{\beta}_{mi}.
\end{equation*}
Substituting the Chapman-Enskog expansion~\eqref{modeCE} into Eqs.~\eqref{VstrongVector} and~\eqref{XstrongVector} and collecting the terms at the order of $\delta^0$, we have
\begin{equation}\label{V0}
    \mathbf{A}\cdot\mathbf{V}^{(0)}=2\gamma_mv_zf_{\text{eq}},
\end{equation}
\begin{equation}\label{X0}
    \mathbf{\hat{A}}^{(0)}\cdot\mathbf{X}=2\hat{\gamma}^{(0)}_mu^n,
\end{equation}
Note that the terms of $q$ have been omitted, since they are equal to zero at this level. $\mathbf{\hat{A}}^{(0)}$ and $\hat{\gamma}^{(0)}$ are the result of the integrals of $V^{(0)}_i$. Now, we use $\mathbf{X}$ to express $\gamma_m$. To this end, $u^n$ can be calculated from Eq.~\eqref{X0} as $\hat{\gamma}^{(0)}_m\neq0$,
\begin{equation}
    u^n=\frac{1}{2}\left(\hat{\gamma}^{(0)}_m\right)^{-1}\mathbf{\hat{A}}^{(0)}\cdot\mathbf{X},
\end{equation}
substituting which into $\gamma_m$, we have
\begin{equation}
\begin{aligned}
    \gamma_m=\int X^n_mu^n\mathrm{d}\bm{x}\mathrm{d}\theta&=\frac{1}{2}\left(\hat{\gamma}^{(0)}_m\right)^{-1}\mathbf{\hat{A}}^{(0)}\cdot\int X^n_m\mathbf{X}\mathrm{d}\bm{x}\mathrm{d}\theta\\
    &=\frac{1}{2}\left(\hat{\gamma}^{(0)}_m\right)^{-1}\mathbf{\hat{A}}^{(0)}\cdot\mathbf{A}^{\ast},
\end{aligned}
\end{equation}
where
\begin{equation}
    [\mathbf{A}^{\ast}]_i=\begin{cases}
        \alpha_{mi},&\quad i<m,\\
        \alpha^{\ast}_{mm}=\int X^n_mX^{n+1}_m\mathrm{d}\bm{x}\mathrm{d}\theta,&\quad i=m.
    \end{cases} 
\end{equation}
For the PGD fixed-point iteration, if the tolerance satisfies $\mathtt{tol_{\text{in}}}\sim \delta^{-1-\epsilon}$ with $\epsilon\geq0$, there exists a constant $C$ independent of $\delta$ such that
\begin{equation}
    \Big|\alpha^{\ast}_{mm}-\alpha_{mm}\Big|\leq\|X^{n
    +1}_m-X^n_m\|<C\delta^{-1-\epsilon},
\end{equation}
thus
\begin{equation}
    \mathbf{A}^{\ast}=\left(\mathcal{I}+\mathcal{E}\right)\mathbf{A},
\end{equation}
where $\mathcal{I}$ is a unit matrix and $\mathcal{E}$ is defined such that $[\mathcal{E}]_{mm}=C\delta^{-1-\epsilon}$ and the other entries are all zero. Note that $\alpha_{mm}=\|X^n_m\|=1$ due to normalisation. Consequently,
\begin{equation}\label{gamma}
    \gamma_m=\frac{1}{2}\left(\hat{\gamma}^{(0)}_m\right)^{-1}\mathbf{\hat{A}}^{(0)}\cdot\left(\mathcal{I}+\mathcal{E}\right)\mathbf{A}=\frac{1}{2}\left(\hat{\gamma}^{(0)}_m\right)^{-1}\left(\mathcal{I}+\mathcal{E}\right)\mathbf{\hat{A}}^{(0)}\cdot\mathbf{A}.
\end{equation}
Substituting Eq.~\eqref{gamma} into~\eqref{V0},
\begin{equation}
    \mathbf{A}\cdot\mathbf{V}^{(0)}=v_zf_{\text{eq}}\left(\hat{\gamma}^{(0)}_m\right)^{-1}\left(\mathcal{I}+\mathcal{E}\right)\mathbf{\hat{A}}^{(0)}\cdot\mathbf{A},
\end{equation}
thereby
\begin{equation}\label{V0X}
    \mathbf{V}^{(0)}=v_zf_{\text{eq}}\left(\hat{\gamma}^{(0)}_m\right)^{-1}\left(\mathcal{I}+\mathcal{E}\right)\mathbf{\hat{A}}^{(0)},
\end{equation}
since, in general, the difference of these two vectors is not orthogonal to $\mathbf{A}$. Multiplying Eq.~\eqref{V0X} by $\mathbf{X}$ and using the relation~\eqref{X0}, we eventually obtain
\begin{equation}
    \mathbf{X}\cdot\mathbf{V}^{(0)}=X_mV^{(0),n+1}_m+\sum^{m-1}_{i=1}X_iV_i=2v_zf_{\text{eq}}u^n+O(\delta^{-1-\epsilon}).
\end{equation}
For $u^{(0),n+1}_h=\langle v_z,\mathbf{X}\cdot\bm{V}^{(0)}\rangle_{\bm{v}}$, the two statements in Proposition~\ref{prop} are verified.
\end{proof}

\subsection{Treating rarefaction parameter as an extra coordinate}

In practice, parametric analysis is typically performed by calculating flow fields over a wide range of rarefaction degrees. PGD provides a method for rapidly solving the parametrised kinetic equation, where the rarefaction parameter $\delta$ is considered an additional coordinate. Flow properties at any point within the domain of the parameter can be accessed immediately, once the computational vademecum of flow velocity and heat flux is obtained.

The PGD solutions for $h$ now also depend on $\delta$, and can be expressed as
\begin{equation} h\left(x,y,v_r,\theta,v_z,\delta\right)=\sum^m_{i=1}X_i\left(x,y,\theta\right)V_i\left(v_r,v_z,\delta\right).
\end{equation}
The flow velocity and heat flux, hence, are functions of $\delta$, i.e. $u=u\left(x,y,\delta\right)$ and $q=q\left(x,y,\delta\right)$. In contrast to the widely used strategy of introducing parametric modes (additional modes as functions of $\delta$)~\cite{SEVILLA2020112631}, parametrising existing modes $V_i$ can keep the greedy PGD solution process as concise as possible. The algorithm for obtaining the PGD modes is the same as the Algorithm~\ref{Algorithm}, except for the following modifications. 

Rather than a single value input, the parameter $\delta$ in Eqs.~\eqref{Vstrong} and~\eqref{macroEq} is an independent variable that covers a continuous range. When solving the equations, it is discretised by $N_{\delta}$ discrete nodes $\{\delta^j\}^{N_{\delta}}_{j=1}$. Due to the dependency on the parameter, the equation for velocity modes is modified as 


\begin{equation}~\label{VstrongPara}
    \begin{aligned}
       \delta \bar\alpha_{mm} V_m^{n+1} + v_r \bar\beta_{mm}V_m^{n+1}=2\delta \bar\gamma_m(\delta)v_zf_{\text{eq}}+\frac{4}{15}\delta\bar\kappa_m(\delta)v_z(\bm{v}^2-\frac{5}{2}) f_{\text{eq}} \\ 
       -\bar\sigma_m v_z f_{\text{eq}}- \sum_{i=1}^{m-1}\left(\delta \bar\alpha_{mi}V_i + v_r\bar\beta_{mi}V_i\right),
    \end{aligned}
\end{equation}
where
\begin{equation*}
    \begin{aligned}
        \bar \alpha_{mi}=\int X^{n}_{m}X_{i}\mathrm{d}\bm{x}\mathrm{d}\theta,\quad
        \bar\beta_{mi}=\int\left(X^n_m\nabla\cdot\bm{s}X_i\right) \mathrm{d}\bm{x}\mathrm{d}\theta,\\
        \bar\gamma_{m}=\int X^n_{m}u^n\mathrm{d}\bm{x}\mathrm{d}\theta,\quad \bar \kappa_{m}=\int X^n_{m}q^n\mathrm{d}\bm{x}\mathrm{d}\theta,\quad\bar \sigma_m=\int X^n_{m}\mathrm{d}\bm{x}\mathrm{d}\theta,\\
        X_i=X^n_m,\quad \text{for }i=m.
    \end{aligned}
\end{equation*}

For the boundary value problem of spatio-angular modes,

\begin{equation}~\label{XstrongPara}
    \begin{aligned}
       \tilde{\alpha}_{mm} X^{n+1}_m+ \tilde{\beta}_{mm}\nabla\cdot\bm{s}X^{n+1}_m=&2\tilde \gamma_m(x,y)+\frac{4}{15}\tilde\kappa_m(x,y)-\tilde\sigma_m \\
        &-\sum^{m-1}_{i=1}\left(\tilde{\alpha}_{mi} X_i+\tilde{\beta}_{mi}\nabla\cdot\bm{s}X_i\right),\quad\left(x,y\right)\in\Omega,\\
       X^{n+1}_m=0,\quad \left(x,y\right)\in\partial\Omega,\ \bm{s}\cdot\bm{n}<0,
    \end{aligned}
\end{equation}
where
\begin{equation*}
    \begin{aligned}
        \tilde{\alpha}_{mi}=\int\delta V_{m}V_{i}v_r\mathrm{d}v_r\mathrm{d}v_z\mathrm{d}\delta,\quad \tilde{\beta}_{mi}=\int v_rV_{m}V_{i}v_r\mathrm{d}v_r\mathrm{d}v_z\mathrm{d}\delta,\\
        \tilde{\gamma}_m=\int \delta v_zV_{m}f_{\text{eq}}v_r u^n\mathrm{d}v_r\mathrm{d}v_z\mathrm{d}\delta,\quad\tilde{\kappa}_m=\int\delta v_z\left(\bm{v}^2-\frac{5}{2}\right)V_{m}f_{\text{eq}}v_rq^n\mathrm{d}v_r\mathrm{d}v_z\mathrm{d}\delta,\\
        \tilde{\sigma}=\int sV_{m}v_r\mathrm{d}v_r\mathrm{d}v_z\mathrm{d}\delta.
    \end{aligned}
\end{equation*}

Note that in equations~\eqref{VstrongPara} and  ~\eqref{XstrongPara}, PGD coefficients on the right-hand sides may depend on spatial coordinates or parameter. In Algorithm~\ref{Algorithm}, the macro equations are solved on each $\delta$ collocation point. 

For the parametric problem, the computational complexity increases to $O\left(M_{\text{md}} N^{\text{in}}_{\text{itr}}\left(N_rN_zN_{\delta}+ N_{\theta}N_{\text{DG}}+2N_{\text{CG}}N_{\delta}\right)\right)$ and the required memory becomes $O\left(M_{\text{md}}\left(N_rN_zN_{\delta}+ N_{\theta}N_{\text{DG}}+2N_{\text{CG}}N_{\delta}\right)\right)$.

\section{Numerical examples and Discussions}\label{Sec:Result}

This section presents numerical examples of Poiseuille flows through long, straight channels with square, trapezoidal, and circular cross sections to demonstrate the performance of PGD solutions for this high-dimensional, parametrised problem governed by the system~\eqref{ShakhovSIS} and~\eqref{macroEq}. Figure~\ref{Mesh} illustrates the two-dimensional spatial domains and the triangular meshes for the discontinuous/continuous Galerkin finite element approximation. The side length, lower base, and radius of the three cross sections are set as the characteristic flow length. The meshes contain 128, 128, and 244 triangles, respectively. The spatial polynomial degree is $k=3$. The truncated velocity domains $v_r\in[0,4]$ and $v_z\in[-4,4]$ are discretised by $N_r=N_z=24$ non-uniformly distributed nodes~\cite{SUImplicitDG}. The polar angle $\theta\in[0,2\pi]$ is divided into $N_{\theta}=48$ intervals of equal spacing. Integrations involving velocity and angle are approximated by using the midpoint rule. For comparison, the same resolution is used for the full-rank results. When considering the parametric problem, the range of the rarefaction parameter $\delta\in[0.01,100]$ is discretised by $N_{\delta}=33$ nodes, and the collocation method is used; thus the corresponding integration is replaced by a sum of the function evaluated at the discrete points. {\color{blue}For the PGD algorithm, the mode enrichment is stopped when $\mathtt{res}_{\text{out}}<\mathtt{tol}_{\text{out}}=5\times10^{-3}$ (or the number of modes reaches $M_{\text{md}}=20$), and the tolerance for the fixed-point iteration is $\mathtt{tol}_{\text{in}}=\min\{10^{-1},\delta^{-1}\}$, which depends on $\delta$ when $\delta\leq0.1$ as it determines the AP property of the low-rank solver.} For a full-rank solution, the synthetic iterative scheme is stopped with tolerance
\[
\max\Bigg\{\frac{\|u^{n+1}-u^n\|}{\|u^n\|},\frac{\|q^{n+1}-q^{n}\|}{\|q^n\|}\Bigg\}<10^{-5}.
\] {\color{magenta}The convergence rate (i.e., the maximum eigenvalue of the iteration matrix) of the synthetic iterative scheme is small in a wide range of rarefaction, and this criterion is enough to guarantee true convergence~\cite{GSISsiam}.} All tests are run on a single Intel\textregistered~core\texttrademark~i7-12700K CPU. 


\begin{figure}[t]
\centering
\includegraphics[width=\textwidth]{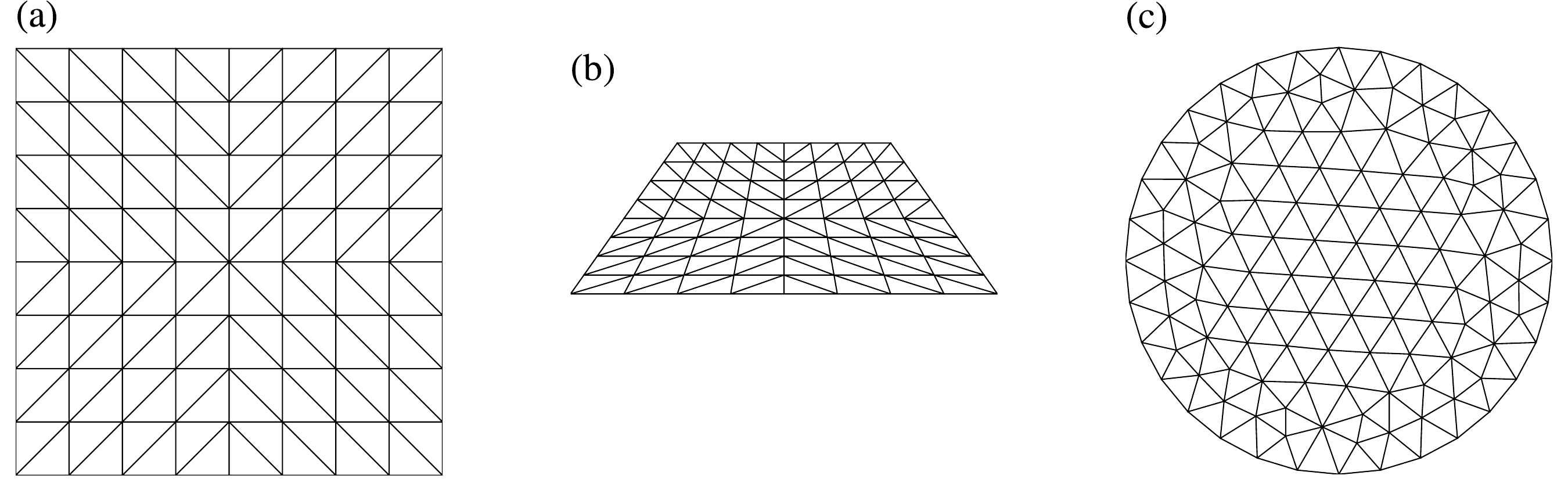}
\caption{Geometry and triangular mesh for two-dimensional spatial domain. (a) Square domain partitioning by 128 triangular elements; (b) Isosceles trapezoidal domain with a ratio of bases of 0.5 and an acute angle of 54.74\textdegree, partitioning by 128 elements; (c) Circular domain partitioning by 244 elements.}\label{Mesh}
\end{figure}

The results are shown based on flow velocity and heat flux calculated from moments of $h$, see Eq.~\eqref{moments}. In particular,
\begin{equation}
    u_h=\sum^m_{i=1}Y_iU_i,\quad q_h=\sum^m_{i=1}Y_iQ_i,
\end{equation}
where
\begin{equation*}\label{PGDmoments}
    Y_i=\int X_i\mathrm{d}\theta,\ U_i=\int v_zV_i\mathrm{d}\bm{v}_{rz},\ Q_i=\int v_z\left(\bm{v}^2-\frac{5}{2}\right)V_i\mathrm{d}\bm{v}_{rz}.
\end{equation*}
are called the $i$-th macroscopic spatial, velocity, and heat flux modes. We also name $Y_i\times U_i$ and $Y_i\times Q_i$ the modes of the velocity and heat flux fields, respectively.

\subsection{Flow through channel with square cross section}

\begin{figure}[t]
\centering
\includegraphics[width=\textwidth]{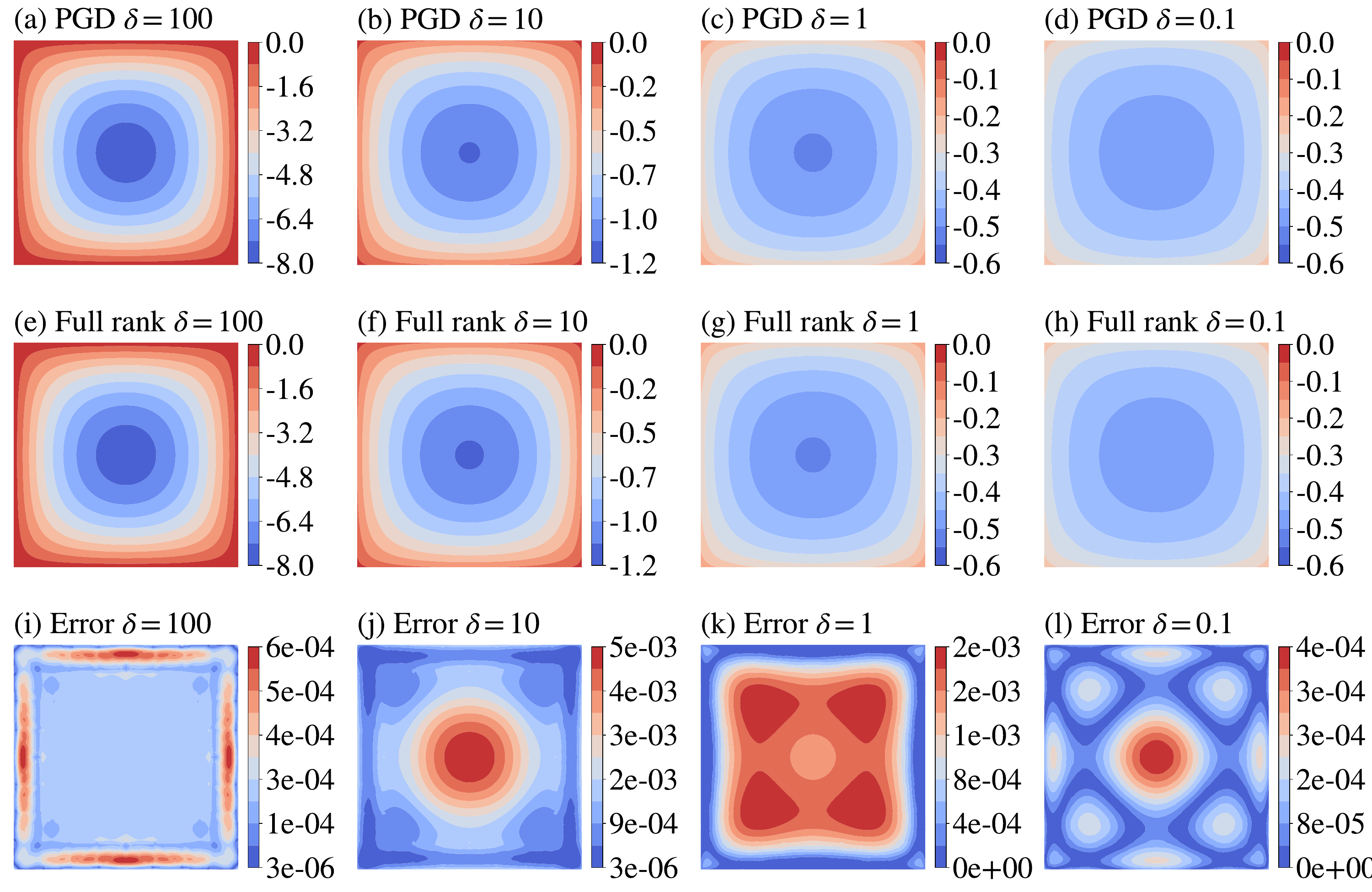}
\caption{Velocity contours of Poiseuille flow through a channel with square section from low-rank PGD and full-rank solutions.}\label{GSISU}
\end{figure}

\begin{figure}[t]
\centering
\includegraphics[width=\textwidth]{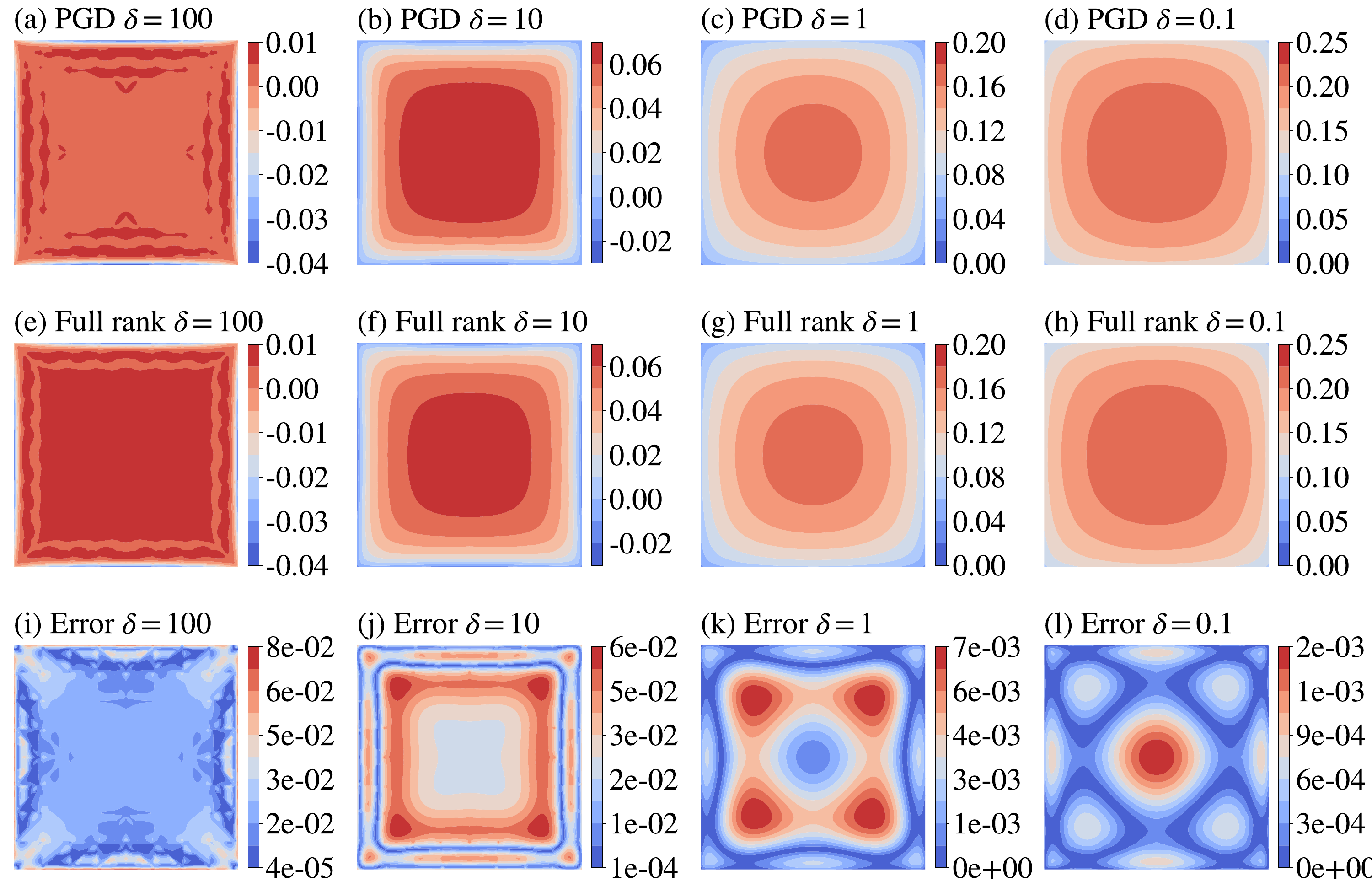}
\caption{Heat flow contours of Poiseuille flow through a channel with square section from low-rank PGD and full-rank solutions.}\label{GSISQ}
\end{figure}

Figures~\ref{GSISU} and~\ref{GSISQ} illustrate the macroscopic flow velocity and heat flux of the Poiseuille flow through a long channel with a square cross section. Comparison between PGD solutions and full-rank solutions is visualised with the contour of relative error
\[
\frac{|u_{\text{PGD}}-u_{\text{full-rank}}|}{\max|u_{\text{full-rank}}|},\quad \frac{|q_{\text{PGD}}-q_{\text{full-rank}}|}{\max|q_{\text{full-rank}}|}.
\]
Four values of the rarefaction, $\delta=100$, $\delta=10$, $\delta =1$ and $\delta=0.1$ are considered. Generally, PGD recovers the flow field accurately to within $1\%$ for velocity flow and $10\%$ for heat flow. {\color{magenta}The PGD solutions at different values of $\delta$ are compared with the full-rank solutions and reference data from the literature and theoretical models in Table~\ref{Table1},} where the flow rates $G_P$ and $G_T$, the errors of the PGD results (relative to the full-rank ones), the number of iteration steps for the synthetic iterative scheme to obtain a full-rank solution, the number of modes in the low-rank solution (i.e., the rank of the solution), as well as the CPU times consumed by the full-rank and the low-rank solvers are tabulated. It is shown that PGD can produce almost the same flow rates of Poiseuille flows (i.e., $G_P$) as full-rank results when $\delta$ is large or becomes small. The precision of low-rank solutions slightly deteriorated when $\delta$ is moderate; however, the maximum difference for velocity flow rates (which appears in $G_P$ at $\delta=10$) is about $6.5\%$. For flow rate $G_T$, the PGD results are accurate in the rarefied regime; as $\delta$ increases, $G_T$ decreases, and the errors for PGD solutions become large, where the maximum difference is 18\% for $\delta=1000$. However, it should be mentioned that $G_T$ in the near-continuum regime approaches zero, and a small absolute discrepancy produces a large relative error. Taking into account the computational cost, the number of iteration steps for a full-rank solution increases as $\delta$ increases, so the CPU time also increases. For a low-rank solution, the rank of the solution is generally small (the enrichment stops when $\mathtt{res}_{\text{out}}<5\times10^{-3}$); in particular, the rank is no greater than 3, when $\delta=0.001$ and $\delta=1000$, which is consistent with the analysis in $\S$~\ref{subsec:hydro} and $\S$~\ref{subsec:free}. That is, the proposed PGD scheme can capture the intrinsic low-rank structure of a kinetic solution in the continuum and free-molecular regimes. Since solutions can be represented by a few PGD modes, the low-rank solver can significantly reduce the computational cost. Compared to the full-rank solver, the largest reduction in CPU time is a factor of 500, and the required memory is only $0.85\%$. On average, storing the full-rank solution requires 35 million double-precision numbers at about 283 MB, while for the low-rank solution it requires 0.31 million doubles occupying 2.4 MB memory. 

The PGD low-rank solver has significantly reduced computational cost, which, however, introduces a trade-off: as the solution rank decreases, time and memory savings increase, but accuracy degrades. Therefore, an interesting question arises: What is the CPU time cost of the low-rank solver when it produces almost the same result as the full-rank solver? Figures~\ref{converge_vel} and~\ref{converge_heat} illustrate the convergence of flow rates with respect to computational time for different values of $\delta$. It is shown that the PGD solution can become more accurate by adding more modes, with an advantage in CPU time over the full-rank solver at the same accuracy.

Finally, to quantify the behaviour of the modes in the low-rank solutions, the relative amplitudes (relative to those with maximum amplitude) of the first 15 modes of the velocity and heat flux fields are plotted in Figure~\ref{modeAmp}, which are sorted in descending order. The results are compared with the modes obtained by applying SVD to the full-rank solution, where SVD provides the optimal choice for the reduced basis~\cite{PGDprimer}. It is observed that the SVD results decay rapidly with the mode index, indicating the low-rank structure of the solutions. The magnitude and downward trend of the PGD mode amplitudes are generally close to those of the SVD results, showing good performance of the PGD algorithm in constructing a reduced basis.

\begin{figure}[t]
\centering
\includegraphics[width=\textwidth]{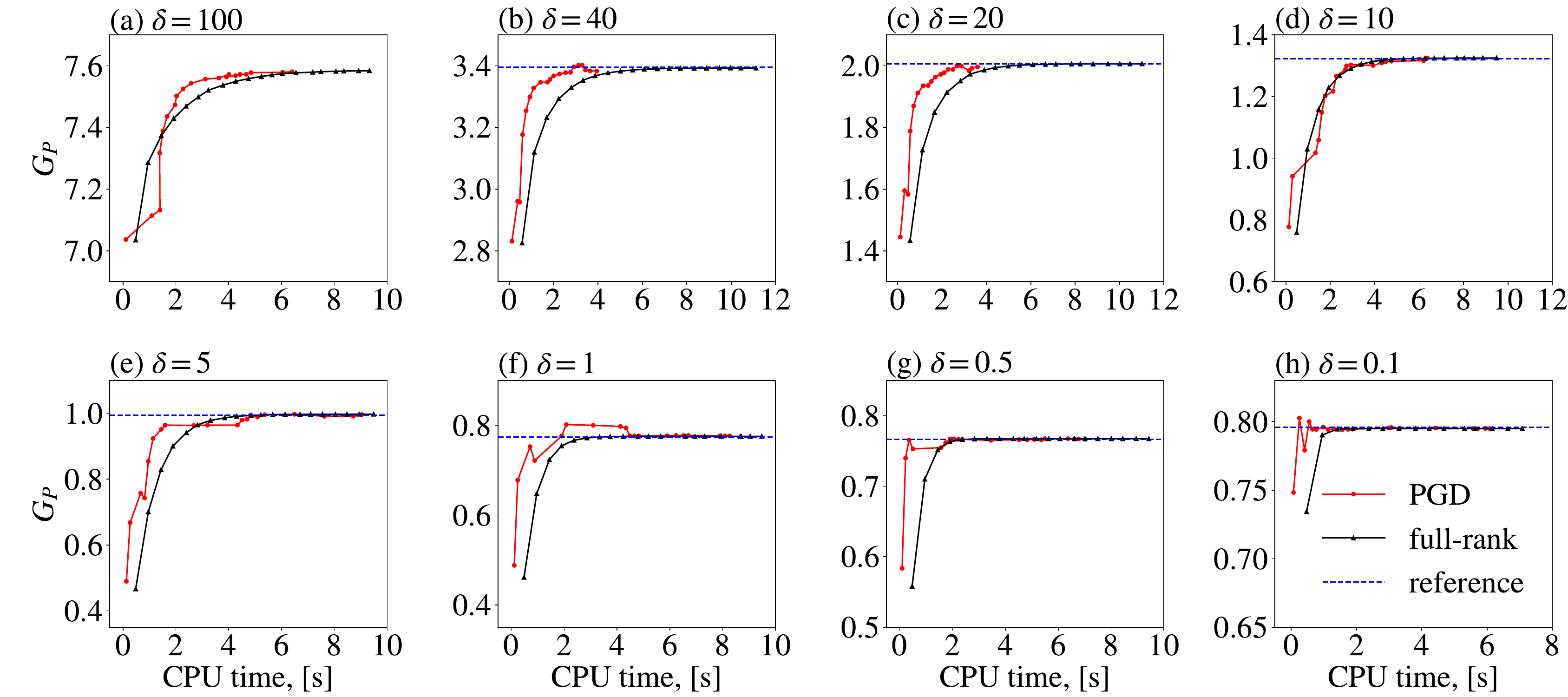}
\caption{Convergence of $G_P$ with respect to computation time for different values of $\delta$. Data are collected after each iteration of the full-rank solution and each enrichment of the PGD solution, until reaching 20 steps for both. Reference data are collected from~\cite{Sharipov1999}.}\label{converge_vel}
\end{figure}

\begin{figure}[t]
\centering
\includegraphics[width=\textwidth]{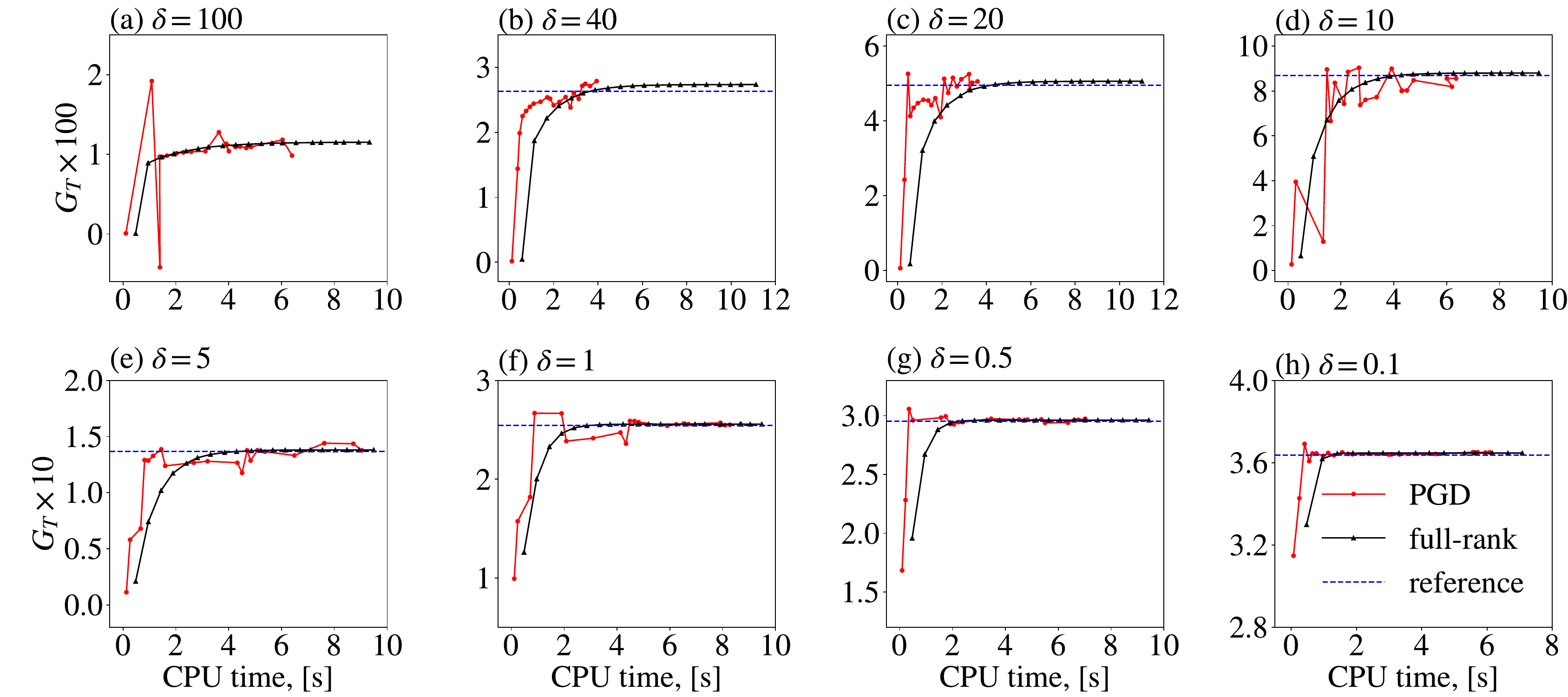}
\caption{Convergence of $G_T$ with respect to computation time for different values of $\delta$. Data are collected after each global iteration of the full-rank solution and each enrichment of the PGD solution, until reaching 20 steps for both. Reference data are collected from~\cite{Sharipov1999}. The results are scaled by 100 or 10 times.}\label{converge_heat}
\end{figure}

\begin{figure}[h]
\centering
\includegraphics[width=\textwidth]{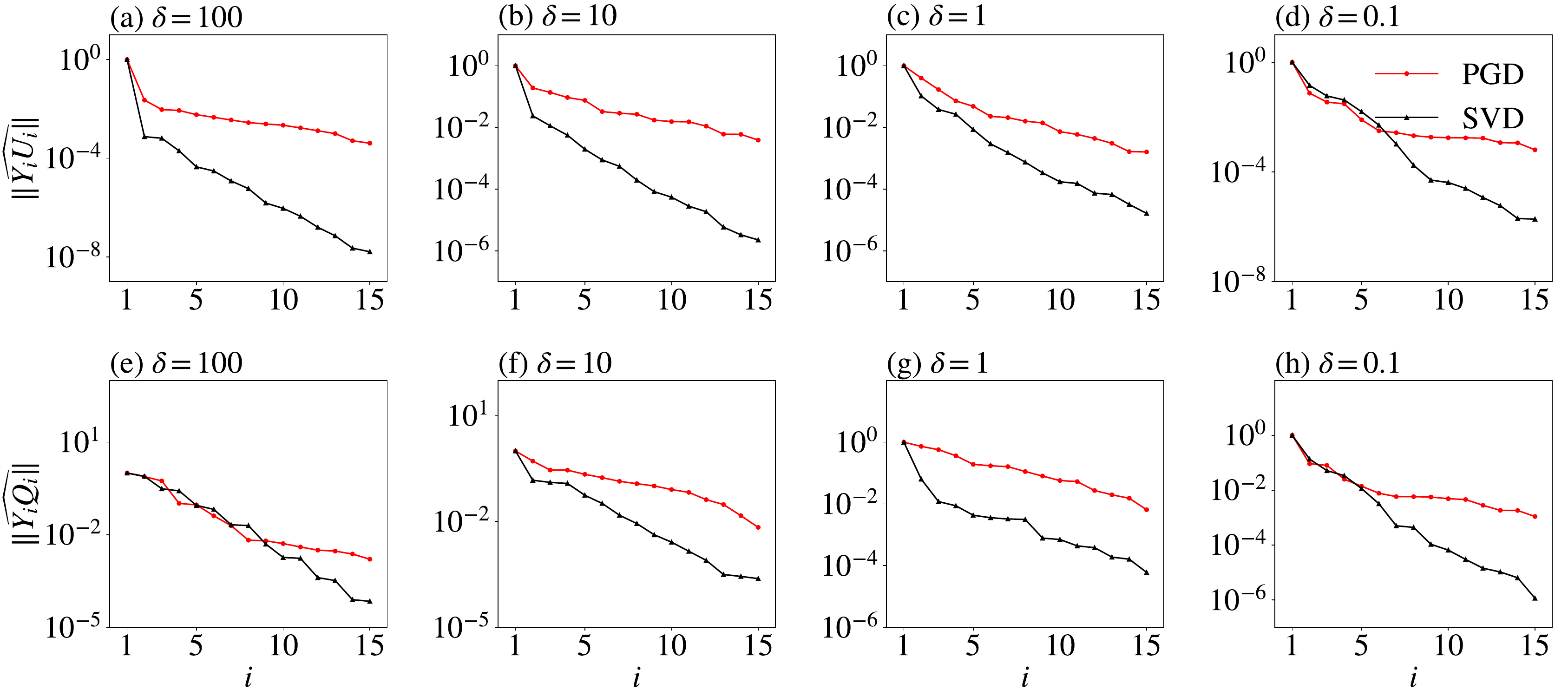}
\caption{Relative amplitudes of the first 15 modes of velocity and heat flux fields: $\|\widehat{Y_iU_i}\|=\|Y_i\|\cdot\|U_i\|/\max_j\{\|Y_j\|\cdot\|U_j\|\}$ and $\|\widehat{Y_iQ_i}\|=\|Y_i\|\cdot\|Q_i\|/\max_j\{\|Y_j\|\cdot\|Q_j\|\}$. Red lines with circular markers show PGD modes. Black lines with triangular marks are the modes obtained by SVD for the full-rank solutions. The modes are sorted by amplitude in descending order. }\label{modeAmp}
\end{figure}
 
\begin{sidewaystable}[t]
\centering
\begin{tabular}{ccccccccccccccc}
\hline
\multirow{2}{*}{$\delta$} & \multicolumn{2}{c}{Reference~\cite{Sharipov1999}} & & \multicolumn{4}{c}{Full-rank} & & \multicolumn{6}{c}{PGD} \\
\cline{2-3}\cline{5-8}\cline{10-15}
 & $G_P$ & $G_T$ & & $G_P$ & $G_T$ & $N^{\text{fr}}_{\text{itr}}$ & $t_{\text{CPU}}$, [s] & & $G_P$ & $Err, [\%]$ & $G_T$ & $Err, [\%]$ & $m$ & $t_{\text{CPU}}$, [s] \\
\hline
0.001 & 0.8317 & 0.4181 & & 0.8366 & 0.4172 & 3 & 1.66 & & 0.8362 & 0.05 & 0.4170 & 0.05 & 3 & 0.26\\
0.1 & 0.7959 & 0.3637 & & 0.7949  & 0.3647  &  6 & 3.14 & & 0.7948 & 0.01  & 0.3647 & 0.00 & 15 & 3.13\\
0.5 & 0.7663 & 0.2953 & & 0.7674  & 0.2962  & 10 & 5.19 & & 0.7669 & 0.07  & 0.2933 & 0.98 & 7 & 0.83\\
1 & 0.7739 & 0.2545 & & 0.7754  & 0.2557  & 14 & 7.45 & & 0.7702 & 0.67 & 0.2469 & 3.4 & 9 & 2.11\\
2 & 0.8175 & 0.2070 & & 0.8194  & 0.2084  & 18 & 9.16 & & 0.8069 & 1.5 & 0.1918 & 8.0 & 8 & 2.14\\
5 & 0.9950 & 0.1366 & & 0.9973  & 0.1381  & 22 & 11.46 & & 0.9634 & 3.4 & 0.1250 & 9.5 & 9 & 2.47\\
10& 1.323 & 0.0868 & & 1.325  & 0.0880  & 22 & 11.56 & & 1.239 & 6.5 & 0.0777 & 12 & 7 & 1.59\\
20& 2.006 & 0.0495 & & 2.006  & 0.0506  & 23 & 12.00 & & 1.935 & 3.5 & 0.0456 & 9.9 & 7 & 1.21\\
40& 3.395 & 0.0263 & & 3.393  & 0.0273  & 25 & 13.05 & & 3.346 & 1.4 & 0.0247 & 9.5 & 8 & 1.47\\
50& - & - & & 4.090  & 0.0223  & 27 & 14.19 & & 4.040 & 1.2 & 0.0200 & 10 & 8 & 1.39\\
80& - & - & & 6.186  & 0.0143  & 33 & 16.46 & & 6.118 & 1.1 & 0.0127 & 11 & 8 & 1.63\\
100& - & - & & 7.585 & 0.0115  & 37& 19.32 & & 7.473 & 1.5 & 0.0100 & 13 & 7 & 1.88\\ 
1000 & - & - & & 70.80 & 0.0011 & 202 & 110.77 & & 70.29 & 0.7 & 0.0009 & 18 & 2 & 0.21 \\
\hline
\multicolumn{10}{l}{$\delta=1000,\ G_P=70.29,\ G_T=0$ from Eq.~\eqref{GpHydro}}\\
\multicolumn{10}{l}{$\delta =0,\ G_P=0.8387,\ G_T=0.4194$ from Eq.~\eqref{GpGtfree}}\\
\end{tabular}
\caption{Comparison of full-rank and PGD solutions at different rarefaction parameters: $G_P$ -- Poiseuille flow rate; $G_T$ -- thermal creep flow rate; $N^{\text{fr}}_{\text{itr}}$ -- number of iterations to obtain full-rank solutions; $m$ -- total number of PGD modes that represent a low-rank solution; $t_{\text{CPU}}$ -- CPU time measured in second; $Err$ -- the relative errors comparing the PGD results with the full-rank solutions. Results for some $\delta$ in~\cite{Sharipov1999} are also listed}\label{Table1}
\end{sidewaystable}

\subsection{Flow through channel with trapezoidal/circular cross section}

\begin{figure}[h]
\centering
\includegraphics[width=0.85\textwidth]{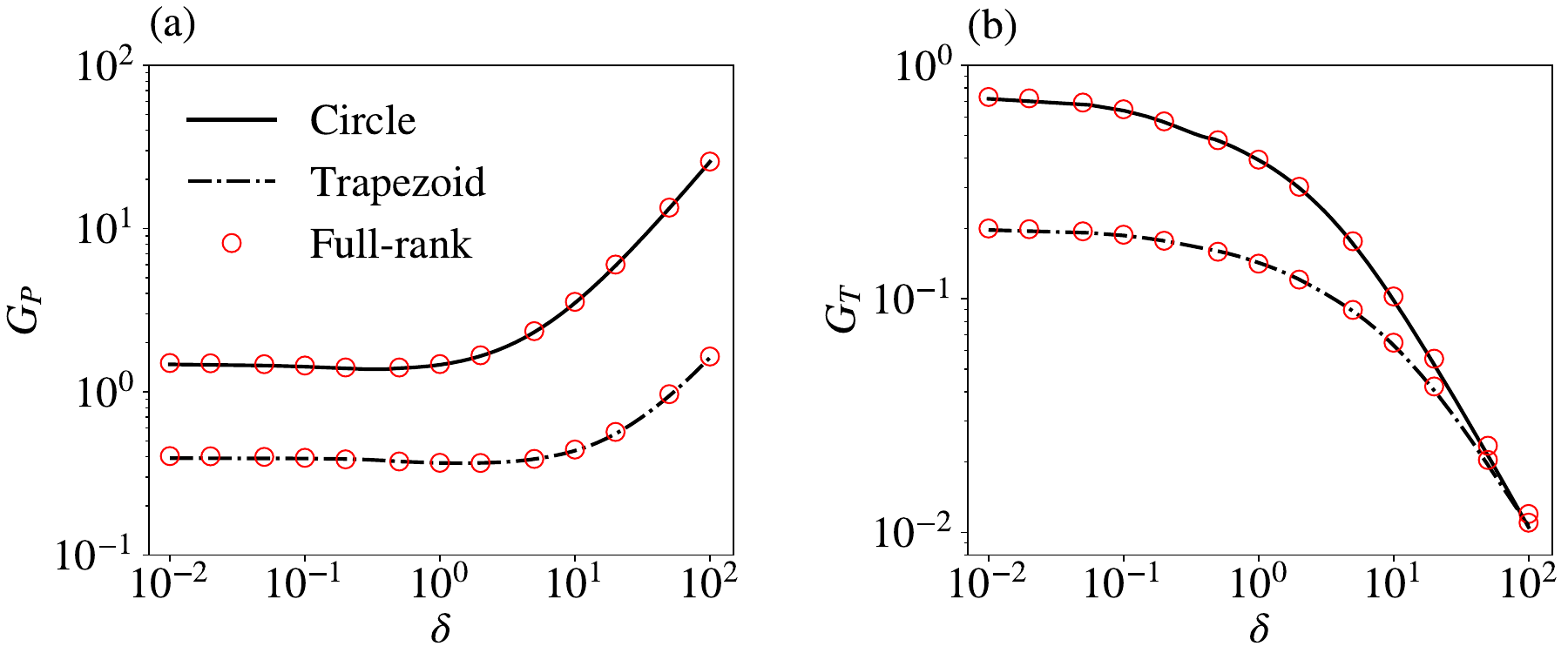}
\caption{(a) Dimensionless flow rate of Poiseuille flow and (b) dimensionless flow rate of thermal creep flow through long channels with trapezoidal and circular cross sections. Lines are PGD general solutions as continuous functions of rarefaction parameter $\delta$. Markers are full-rank solutions at some values of $\delta$.}\label{FRs}
\end{figure}

The parametrised kinetic equation~\eqref{ShakhovCylinder} with $\delta$ serving as an additional coordinate is solved by PGD for flows through long channels with trapezoidal and circular cross sections. A single PGD run produces a general solution of the flow field that contains all possible data for $\delta$ varying from 0.01 to 100. Figure~\ref{FRs} illustrates dimensionless flow rates as continuous functions of the rarefaction parameter $\delta$. Full-rank solutions calculated at some values of $\delta$ are also included for comparison. PGD is found to be able to recover solutions with high precision over the whole range of $\delta$ considered. In addition, the computation takes 0.220 hours to obtain the general solution for the channel with the trapezoidal cross section and 0.278 hours for the channel with the circular cross section.     

Once the PGD general solution is obtained and stored as a computational vademecum, it can be reused for, for example, parametric analysis and the investigation of inverse problems. Here, as an example, the thermomolecular pressure difference (TPD) problem is considered. In this state, the Poiseuille and thermal creep flows counterbalance and the net flow rate vanishes. The pressures and temperatures in the two reservoirs connected by the long channel satisfy the following relationship
\begin{equation}
    \frac{P_2}{P_1}=\left(\frac{T_2}{T_1}\right)^{\eta},
\end{equation}
where $\eta$ is the TPD coefficient. To evaluate this coefficient, it is assumed that the temperature ratio $T_2/T_1$ and the rarefaction $\delta_1$ with respect to $P_1$ and $T_1$ are known. The pressure ratio $P_2/P_1$ is solved from~\cite{Graur2009,Ritos01112011}
\begin{equation}\label{TPD}
    \frac{\mathrm{d}\tilde{P}}{\mathrm{d}\tilde{T}}=\frac{\tilde{P}}{\tilde{T}}\frac{G_T\left(\delta\right)}{G_P\left(\delta\right)},
\end{equation}
where $\tilde{P}=P/P_1$ and $\tilde{T}=T/T_1$ are the normalised local pressure and temperature along the flow direction, and $\delta=\delta_1\tilde{P}/\tilde{T}$ is the local rarefaction parameter. Using the PGD solutions of $G_P\left(\delta\right)$ and $G_T\left(\delta\right)$, $P_2/P_1$ can be obtained by integrating Eq.~\eqref{TPD} along $1\le\tilde{T}\le T_2/T_1$. The TPD coefficients for $T_2/T_1=3.8$ and various $\delta_1$ are listed in Table~\ref{Table2}. The results obtained from the general PGD solutions are compared with those in previous works based on full-rank calculations. This confirms the accuracy of the parametrised solutions obtained by PGD.

\begin{table}[t]
\centering
\begin{tabular}{cccccc}
\hline
$\delta_1$ & Circle & Circle in~\cite{Graur2009} & $\delta^{\ast}_1$ & Trapezoid & Trapezoid in~\cite{Ritos01112011}\\
\hline
0.02 & 0.4779 & 0.4849 & 0.02 & 0.4872 & 0.4907 \\
0.05 & 0.4729 & 0.4706 & 0.05 & 0.4744 & 0.4791 \\
0.1  & 0.4645 & 0.4531 & 0.1  & 0.4578 & 0.4646 \\
0.2  & 0.4473 & 0.4275 & 0.2  & 0.4439 & 0.4441 \\
0.5  & 0.3977 & 0.3761 & 0.5  & 0.4314 & 0.4019 \\
0.8  & 0.3587 & 0.3389 & 0.8  & 0.3949 & 0.3722 \\
1.0  & 0.3377 & 0.3186 & 1.0  & 0.3728 & 0.3560 \\
2.0  & 0.2637 & 0.2468 & 2.0  & 0.3007 & 0.2961 \\
5.0  & 0.1496 & 0.1423 & 5.0  & 0.1977 & 0.1991 \\
10.0 & 0.0764 & 0.0743 & 10.0 & 0.1241 & 0.1246 \\
20.0 & 0.0316 & 0.0306 & 20.0 & 0.0643 & 0.0647 \\
\hline
\end{tabular}
\caption{TPD coefficients for $T_2/T_1=3.8$ and various $\delta_1$. Results from PGD parametrised solutions are compared with those in previous works based on full-rank calculations. $\delta^{\ast}_1=0.4483\delta_1$ is the rarefaction parameter based on the hydraulic diameter of the channel.}\label{Table2}
\end{table}

\clearpage

\section{Conclusions}\label{Sec:conclusion}

An asymptotic-preserving reduced-order solution strategy based on the PGD method is proposed for the high-dimensional parametrised Shakhov model equation, which describes the dynamics of rarefied gas flows driven by pressure and temperature gradients through a long channel with an arbitrary cross section. The low-rank solution of the velocity distribution function is approximated as the sum of a small number of function products, where each product contains a spatio-angular mode and a mesoscopic velocity mode. A sophisticated algorithm is designed to solve these PGD modes. Due to the separated representation, the evaluation of the velocity distribution function is transformed into a few low-dimensional problems. Consequently, computational complexity is significantly reduced. Furthermore, grouping the spatial coordinates with the polar angle allows the method to treat arbitrary two-dimensional geometries. 

The AP property is achieved on the basis of the synthetic iterative scheme, where synthetic approximations for the moments appearing in the collision operator of the kinetic equation are obtained from macroscopic equations and inserted into the PGD iteration. The scheme becomes a macroscopic solver for the Navier-Stokes equations, and hence it can capture the correct low-rank solution when the Knudsen number is small. The numerical results show that the low-rank method can produce a solution in several seconds at any Knudsen number and costs only 0.85\% of the memory for the full-rank solution. The method can achieve up to 500 times speedup over the full-rank solver. For parametrised flows, the rarefaction parameter serves as an additional coordinate. A general solution is calculated once and for all in the whole parameter range, from which a specific solution is accessible in real time. Compared to full-rank solutions, the general PGD solution agrees with the full-rank flow rates over the range of $\delta$ considered.


The proposed reduced-order method can be further extended to other linear kinetic transport processes, such as those for phonons, neutrons and photons. In each case the unknown is a distribution in space and in a mesoscopic variable, and the dimension can be reduced by similar separated-representation strategy. The AP property can be carried over once the synthetic equation is replaced by the macroscopic limit appropriate to the collision operator. The scale parameters can still be treated as additional coordinate and computed offline, while, however, the separated representation and synthetic equations may have to be rebuilt according to specific collision operator and coordinate splitting. The Shakhov long channel problem treated in this work does not by itself supply those formulations. 

\section*{Acknowledgments}

We are deeply grateful to Dr Xi Zou (Zienkiewicz Institute for Modelling, Data and AI, Swansea University) for his discussions on the principle of PGD. His perspective prompted the creation of our method.



\bibliographystyle{elsarticle-num} 
\bibliography{Ref}

\end{document}